\pdfoutput=1
\RequirePackage{ifpdf}
\ifpdf % We~are running pdfTeX in pdf mode
\documentclass[pdftex]{sigma}
\else
\documentclass{sigma}
\fi

\newtheorem{Theorem}{Theorem}[section]
\newtheorem{Question}[Theorem]{Question}
\newtheorem{Corollary}[Theorem]{Corollary}

\newtheorem{Proposition}[Theorem]{Proposition}

 { \theoremstyle{definition}
\newtheorem{Definition}[Theorem]{Definition}
\newtheorem{Example}[Theorem]{Example}
\newtheorem{Remark}[Theorem]{Remark} }

\begin{document}
%\allowdisplaybreaks

\newcommand{\arXivNumber}{2001.04087}

\renewcommand{\thefootnote}{}

\renewcommand{\PaperNumber}{013}

\FirstPageHeading

\ShortArticleName{Curvature-Dimension Condition Meets Gromov's $n$-Volumic Scalar Curvature}

\ArticleName{Curvature-Dimension Condition Meets\\ Gromov's $\boldsymbol{n}$-Volumic Scalar Curvature\footnote{This paper is a~contribution to the Special Issue on Scalar and Ricci Curvature in honor of Misha Gromov on his 75th Birthday. The full collection is available at \href{https://www.emis.de/journals/SIGMA/Gromov.html}{https://www.emis.de/journals/SIGMA/Gromov.html}}}

\Author{Jialong DENG}

\AuthorNameForHeading{J.~Deng}

\Address{Mathematisches Institut, Georg-August-Universit\"at, G\"ottingen, Germany}
\Email{\href{mailto:jialong.deng@mathematik.uni-goettingen.de}{jialong.deng@mathematik.uni-goettingen.de}}

\ArticleDates{Received July 29, 2020, in final form January 23, 2021; Published online February 05, 2021}

\Abstract{We study the properties of the $n$-volumic scalar curvature in this note. Lott--Sturm--Villani's curvature-dimension condition ${\rm CD}(\kappa,n)$ was showed to imply Gromov's $n$-volumic scalar curvature $\geq n\kappa$ under an additional $n$-dimensional condition and we show the stability of $n$-volumic scalar curvature~$\geq \kappa$ with respect to smGH-convergence. Then we propose a new weighted scalar curvature on the weighted Riemannian manifold and show its properties.}


\Classification{53C23}

\renewcommand{\thefootnote}{\arabic{footnote}}
\setcounter{footnote}{0}

\section{Introduction}

The concept of lower bounded curvature on the metric space or the metric measure space has evolved to a rich theory due to Alexandrov's insight. The stability of Riemannian manifolds with curvature bounded below is another deriving force to extend the definition of the curvature bounded below to a broader space. However, the scalar curvature (of Riemannian metrics) bounded below was yet absent from this picture. Gromov proposed a synthetic treatment of scalar curvature bounded below, which was called the $n$-volumic scalar curvature bounded below, and offered some pertinent conjectures in \cite[Section~26]{Gromov}. Motivated by the ${\rm CD}(\kappa,n)$ condition, we add an $n$-dimension condition to the Gromov's definition and introduce the definition of~${\rm Sc}_{\alpha, \beta}$ on the smooth metric measure space. Details will be given later.

\begin{Theorem}Assume that the metric measure space $(X^n, d, \mu)$ satisfies $n$-dimensional condition and the curvature-dimension condition ${\rm CD}(\kappa, n)$ for $\kappa\geq 0$ and $n\geq 2$, then $(X^n, d, \mu)$ satisfies ${\rm Sc}^{{\rm vol}_n}(X^n)\geq n\kappa$.
\end{Theorem}

\begin{Theorem}
If compact metric measure spaces $(X_i^n, d_i, \mu_i)$ with ${\rm Sc}^{{\rm vol}_n}(X^n_i)\geq \kappa\geq 0$ and SC-radius $r_{x^n_i}\geq R>0$ and $(X_i^n, d_i, \mu_i)$ strongly measured Gromov--Hausdorff converge to the compact metric measure space $(X^n, d, \mu)$ with $n$-dimensional condition, then~$X^n$ also satisfies ${\rm Sc}^{{\rm vol}_n}(X^n)\geq \kappa$ and the SC-radius $r_{X^n}\geq R$.
\end{Theorem}

\begin{Theorem}
 Let $\big(M^n, g, {\rm e}^{-f}\,{\rm dVol}_g\big)$ be the closed smooth metric measure space with ${\rm Sc}_{\alpha, \beta}>0$, then we have the following conclusions:
 \begin{itemize}\itemsep=0pt
 \item[$1.$] If $M^n$ is a spin manifold, $\alpha \in \mathbb{R}$ and $\beta \geq \frac{|\alpha|^2}{4}$, then the harmonic spinors of $M^n$ vanish.
 \item[$2.$] If the dimension $n\geq 3$, $\alpha \in \mathbb{R}$ and $\beta \geq \frac{(n-2)|\alpha|^2}{4(n-1)}$, then there is a metric $\tilde{g}$ conformal to $g$ with positive scalar curvature.
 \item[$3.$] If the dimension $n\geq 3$, $\alpha=2$, $\beta \geq \frac{n-2}{n-1}$ and $\big(N^{n-1},\bar{g}\big)$ is the compact $L_f$-stable minimal hypersurface of $\big(M^n, g, {\rm e}^{-f}\,{\rm dVol}_g\big)$, then there exists a PSC-metric conformal to~$\bar{g}$ on~$N^{n-1}$, where $\bar{g}$ is the induced metric of $g$ on $N^{n-1}$.

 \item[$4.$] Assume $M^n$ is a spin manifold and there exists a smooth $1$-contracting map $h\colon (M^n, g)\to (S^n, g_{\rm st})$ of non-zero degree. If $\alpha \in \mathbb{R}$ , $\beta \geq \frac{|\alpha|^2}{4}$ and ${\rm Sc}_{\alpha, \beta}\geq n(n-1)$, then $h$ is an isometry between the metrics $g$ and $g_{\rm st}$.
\end{itemize}
\end{Theorem}

The paper is organized as follows. In Section~\ref{CD}, we introduce the notions and show that ${\rm CD}(\kappa,n)$ implies ${\rm Sc}^{{\rm vol}_n}\geq (n-1)\kappa$. In Section~\ref{sm}, we show the stability of spaces with ${\rm Sc}^{{\rm vol}_n}\geq\kappa$. In Section~\ref{smooth mm-space}, we present the properties of the smooth metric measure space with ${\rm Sc}_{\alpha, \beta}>0$.

\section[CD meets n-volumic scalar curvature]{CD meets $\boldsymbol{n}$-volumic scalar curvature}\label{CD}
 The $n$-dimensional Aleksandrov space with curvature $\geq\kappa$ equipped with the volume-measure satisfies Lott--Villani--Sturm's weak curvature-dimension condition for dimension $n$ and curvature $(n-1) \kappa$, i.e., ${\rm CD}((n-1)\kappa, n)$, was shown by Petrunin for $\kappa=0$ (and said that for general curvature $\geq\kappa$ the result followed in a similar way) \cite{MR2869253} and then Zhang--Zhu investigated the general case~\cite{MR2747437}. We will modify Gromov's definition of $n$-volumic scalar curvature bounded below in \cite[Section~26]{Gromov} to fill the picture, which means Lott--Sturm--Villani's Ricci curvature $\geq 0$ implies Gromov's scalar curvature $\geq 0$.

 The metric measure space (mm-space) $X=(X,d,\mu)$ means that $d$ is the complete separable length metric on $X$ and $\mu$ is the locally finite full support Borel measure on $X$ equipped with its Borel $\sigma$-algebra. Say that an mm-space $X=(X,d,\mu)$ is locally volume-wise smaller (or not greater) than another such space $X'=(X',d',\mu')$ and write $X<_{\rm vol} X'$ $(X\leq_{\rm vol} X')$, if all $\epsilon$-balls in $X$ are smaller (or not greater) than the $\epsilon$-balls in $X'$, $\mu(B_{\epsilon}(x))<\mu'(B_{\epsilon}(x')) (\mu(B_{\epsilon}(x))\leq \mu'(B_{\epsilon}(x'))$, for all $x\in X, x'\in X'$ and the uniformly small $\epsilon$ which depends on $X$ and $ X'$.

From now on, the Riemannian 2-sphere $\big(S^2(\gamma),d_S, {\rm vol}_S\big)$ is endowed with round metric such that the scalar curvature equal to $2\gamma^{-2}$, $\big(\mathbf{R}^{n-2},d_E, {\rm vol}_E\big)$ is endowed with Euclidean metric with flat scalar curvature and the product manifold $S^2(\gamma)\times \mathbf{R}^{n-2}$ is endowed with the Pythagorean product metrics $d_{S\times E}:=\sqrt{d_S^2 + d_E^2}$ and the volume ${\rm vol}_{S\times E}:={\rm vol}_S \otimes {\rm vol}_E$.

Thus, we have $S^2(\gamma)<_{\rm vol} \mathbf{R}^2 $. If $0<\gamma_1< \gamma_2$, then $S^2(\gamma_1)<_{\rm vol} S^2(\gamma_2)$. Furthermore, $S^2(\gamma)\times \mathbf{R}^{n-2}<_{\rm vol} \mathbf{R}^n$. If $0<\gamma_1< \gamma_2$, then $S^2(\gamma_1)\times \mathbf{R}^{n-2}<_{\rm vol} S^2(\gamma_2)\times \mathbf{R}^{n-2}$.

\begin{Definition}[Gromov's $n$-volumic scalar curvature]
 Gromov's $n$-volumic scalar curvature of $X$ is bounded below by $0$ for $X=(X, d, \mu)$ if~$X$ is locally volume-wise not greater than $\mathbf{R}^n$.

 Gromov's $n$-volumic scalar curvature of $X$ bounds from below by $\kappa>0$ for $X=(X, d, \mu)$ if~$X$ is locally volume-wise smaller than $S^2(\gamma)\times \mathbf{R}^{n-2}$ for all $\gamma>\sqrt{\frac{2}{\kappa}}$, i.e., $X<_{\rm vol} S^2(\gamma)\times \mathbf{R}^{n-2}$ and $\gamma>\sqrt{\frac{2}{\kappa}}$, where $S^2(\gamma)\times \mathbf{R}^{n-2}=\big(S^2(\gamma)\times \mathbf{R}^{n-2},d_{S\times E},{\rm vol}_{S\times E}\big)$.
\end{Definition}

The $n$-volumic scalar curvature is sensitive to the scaling of the measure, but the curvature condition ${\rm CD}(\kappa, n)$ of Lott--Villani--Sturm \cite[Definition~1.3]{MR2237207} is invariant up to scalars of the measure only \cite[Proposition~1.4(ii)]{MR2237207}. Therefore, the $n$-dimensional condition needs to be put into the definition of Gromov's $n$-volumic scalar curvature. In fact, the $n$-dimensional condition is the special case of Young's point-wise dimension in dynamical systems \cite[Theorem~4.4]{MR684248}.

\begin{Definition}[$n$-dimensional condition]
For given positive natural number $n$, the mm-space $X=(X, d, \mu)$ satisfies the $n$-dimensional condition if
\begin{align*}
 \lim\limits_{r \to 0}\frac{\mu(B_r(x))}{{\rm vol}_E(B_r(\mathbf{R}^n))}=1
\end{align*}
 for every $x \in X$, where $B_r(\mathbf{R}^n)$ is the closed $r$-ball in the Euclidean space $\mathbf{R}^n$ and the $B_r(x)$ is the closed $r$-ball with the center $x\in X$.
\end{Definition}

 From now on, the superscript of $n$ in the space $X^n$ means the mm-space $(X^n, d, \mu)$ satisfies $n$-dimensional condition.

 Note that a closed smooth $n$-manifold $M^n$ $(n\geq 3)$ admits a Riemannian metric with constant negative scalar curvature and a Riemannian metric of non-negative scalar curvature which is not identically zero, then by a conformal change of the metric we get a metric of positive scalar curvature according to Kazdan--Warner theorem \cite{MR365409}. Furthermore, if there is a scalar-flat Riemannian metric $g$ on $M^n$, but $g$ is not Ricci-flat metric, then $g$ can be deformed to a~metric with positive scalar curvature according to Kazdan theorem \cite[Theorem~B]{MR675736} or by using Ricci-flow with an easy argument. Hence we will focus more on promoting the positive scalar curvature to positive $n$-volumic scalar curvature.

\begin{Definition}[$n$-volumic scalar curvature]
Assume $X^n=(X^n, d, \mu)$ is the compact mm-space and satisfies the $n$-dimensional condition, we call
\begin{enumerate}\itemsep=0pt
\item[$1.$] the $n$-volumic scalar curvature of $X^n$ is positive, i.e., ${\rm Sc}^{{\rm vol}_n}(X^n)> 0$, if there exists $r_{X^n} > 0$ such that the measures of $\epsilon$-balls in $X^n$ are smaller than the volumes of $\epsilon$-balls in~$\mathbf{R}^n$ for $0 < \epsilon \leq r_{X^n}$.
\item[$2.$] the $n$-volumic scalar curvature of $X^n$ is bounded below by 0, i.e., ${\rm Sc}^{{\rm vol}_n}(X^n)\geq 0$, if there exists $r_{X^n} > 0$ such that the measures of $\epsilon$-balls in $X^n$ are not greater than the volumes of $\epsilon$-balls in~$ \mathbf{R}^{n}$ for $0 < \epsilon \leq r_{X^n}$.

The $r_{X^n}$ is called scalar curvature radius $($SC-radius$)$ of $X^n$ for ${\rm Sc}^{{\rm vol}_n}(X^n)\geq 0$.

\item[$3.$] the $n$-volumic scalar curvature of $X^n$ is bounded below by $\kappa>0$, i.e., ${\rm Sc}^{{\rm vol}_n}(X^n)\geq\kappa>0$, if, for any $\gamma$ with $\gamma> \sqrt{\frac{2}{\kappa}}$, there exists $r_{X^n,\gamma} >0$ such that the measures of $\epsilon$-balls in $X^n$ are smaller than the volumes of $\epsilon$-balls in $S^2(\gamma)\times \mathbf{R}^{n-2}$ for $0 < \epsilon \leq r_{X^n,\gamma}$.

We call $r_{X^n}:=\inf\limits_{\gamma> \sqrt{\frac{2}{\kappa}}} r_{X^n,\gamma}$ is the SC-radius of $X^n$ for ${\rm Sc}^{{\rm vol}_n}(X^n)\geq\kappa>0$.
 \end{enumerate}
\end{Definition}

In particular, we will focus on the case of $\inf\limits_{\gamma> \sqrt{\frac{2}{\kappa}}} r_{X^n,\gamma}\neq 0$ for stability in Section~\ref{sm}.

If the mm-space $X^n$ is locally compact, then the definition of the $n$-volumic scalar curvature bounded below only modifies the definition of the $r_{X^n,\gamma} >0$ to a positive continuous function of~$X^n$.

Two mm-spaces $(X^n, d, \mu)$ and $(X^n_1, d_1, \mu_1)$ are isometric if there exists a one-to-one map $f\colon X^n \to X^n_1$ such that $d_1(f(a),f(b))=d(a,b)$ for $a$ and $b$ are in $X^n$ and $f_*\mu=\mu_1$, where $f_*\mu$ is the push-forward measure, i.e., $f_*\mu(U)=\mu\big(f^{-1}(U)\big)$ for a measureable subset $U\subset X^n_1$. If~$X^n$ satisfies ${\rm Sc}^{{\rm vol}_n}(X^n)\geq\kappa\geq 0$, then each mm-space $(X^n_1, d_1, \mu_1)$ that is isometric to $(X^n, d, \mu)$ also satisfies ${\rm Sc}^{{\rm vol}_n}(X^n_1)\geq\kappa\geq 0$.

\begin{Proposition}Let $g$ be a $C^2$-smooth Riemannian metric on a closed oriented $n$-mani\-fold~$M^n$ with induced metric measure space $(M^n, d_g, {\rm dVol}_g)$, then the scalar curvature of $g$ is positive, ${\rm Sc}_g> 0$, if and only if ${\rm Sc}^{{\rm vol}_n}(M^n)>0 $, and ${\rm Sc}_g\geq\kappa> 0$ if and only if~${\rm Sc}^{{\rm vol}_n}(M^n)\geq\kappa> 0$.
\end{Proposition}

\begin{proof}For a $C^2$-smooth Riemannian metric $g$, one has
\begin{gather*}
{\rm dVol}_g(B_r(x)) = {\rm vol}_E(B_r(\mathbf{R}^n))\left[1 - \frac{{\rm Sc}_g(x)}{6(n+2)} r^2 + O\big(r^4\big)\right]
 \end{gather*}
for $B_r(x) \subset M^n$ as $r\to 0$. Hence $(M^n, d_g, {\rm dVol}_g)$ satisfies the $n$-dimensional condition.

If we have ${\rm Sc}_g > 0$, then, since $M^n$ is compact, there exists $r_{M^n} > 0$, so that ${\rm dVol}_g(B_r(x))< {\rm vol}_E(B_r(\mathbf{R}^n))$ for all $0<r \leq r_{M^n}$. On the other hand, if there exists $r_{M^n} > 0$ such that ${\rm dVol}_g(B_r(x))< {\rm vol}_E(B_r(\mathbf{R}^n))$ for all $0<r \leq r_{M^n}$, then ${\rm Sc}_g$ must be greater than 0.

If ${\rm Sc}^{{\rm vol}_n}(M^n)\geq\kappa> 0$, then ${\rm Sc}_g \geq \kappa > 0$. Otherwise, assume there exist small $\epsilon>0$ such that ${\rm Sc}_g \geq \kappa-\epsilon>0$. That means that there exists a point $x_0$ in $M^n$ such that ${\rm Sc}_g(x_0)= \kappa-\epsilon$, as $M^n$ is compact and the scalar curvature is a continuous function on $M^n$. Thus, we can find a small $r$-ball $B_r(x_0)$ such that the volume of $B_r(x_0)$ is greater than the volume of the $r$-ball in the $S^2(\gamma)\times \mathbf{R}^{n-2}$ for $\gamma=\sqrt{\frac{2}{\kappa-\frac{\epsilon}{2}}}$, which is a contradiction.

On the other hand, ${\rm Sc}_g\geq\kappa> 0$ implies ${\rm Sc}^{{\rm vol}_n}(M^n)\geq\kappa> 0$. Assume ${\rm Sc}_g(x_1)=\kappa$ for some $x_1\in M^n$, then there exists $r_1$ such that ${\rm dVol}_g(B_{r_1}(x))\leq {\rm dVol}_g(B_{r_1}(x_1))$ for $r_1$-balls in $M^n$ and
\begin{gather*}
 {\rm dVol}_g(B_r(x_1)) = {\rm vol}_E(B_r(\mathbf{R}^n))\left[1 - \frac{\kappa}{6(n+2)} r^2 + O\big(r^4\big)\right]
\end{gather*}
 as $r\to 0$. Thus, for any $\gamma$ with $\gamma> \sqrt{\frac{2}{\kappa}}$, there exists $r_{M^n,\gamma} >0$ such that the measures of $\epsilon$-balls in $M^n$ are smaller than the volumes of $\epsilon$-balls in $S^2(\gamma)\times \mathbf{R}^{n-2}$ for $0 < \epsilon \leq r_{M^n,\gamma}$, i.e., ${\rm Sc}^{{\rm vol}_n}(M^n)\geq\kappa> 0$.
\end{proof}

Therefore, we have $S^n_{\frac{\kappa}{n-1}}<_{\rm vol} S^2(\gamma)\times \mathbf{R}^{n-2}$ for all $\gamma>\sqrt{\frac{2}{n\kappa}}$. Here $S^n_{\frac{\kappa}{n-1}}$ is the Riemannian manifold $S^n$ with constant sectional curvature $\frac{\kappa}{n-1}$.

\begin{Remark}For a closed smooth Riemannian manifold $(M^n,g)$, ${\rm Sc}^{{\rm vol}_n}(M^n)\geq 0 $ implies ${\rm Sc}_g\geq 0$. Otherwise, there exists a point in $M^n$ such that the scalar curvature is negative, then the volume of small ball will be greater than the volume of the small ball in Euclidean space, which is a contradiction.

On the other hand, one can consider the case of the scalar-flat metric, i.e., ${\rm Sc}_g\equiv 0$. If $g$ is a strongly scalar-flat metric, meaning a metric with scalar curvature zero such that $M^n$ has no metric with positive scalar curvature, then~$g$ is also Ricci flat according to Kazdan theorem above. Thus, we have
 \begin{gather*}
 {\rm dVol}_g(B_r(x)) = {\rm vol}_E(B_r(\mathbf{R}^n))\left[1 - \frac{\|{\rm Rie}(x)\|^2_g}{120(n+2)(n+4)} r^4 + O\big(r^6\big)\right]
 \end{gather*}
 for $B_r(x) \subset M^n$ as $r\to 0$ {\rm \cite[Theorem~3.3]{MR521460}}. Here ${\rm Rie}$ is the Riemannian tensor. Therefore, if~$g$ is a not flat metric, then $M^n<_{\rm vol} \mathbf{R}^n$. If $g$ is a flat metric, then $M^n\leq_{\rm vol} \mathbf{R}^n$. Thus ${\rm Sc}_g\geq 0$ implies ${\rm Sc}^{{\rm vol}_n}(M^n)\geq 0 $ for a strongly scalar-flat metric $g$.

 However, ${\rm Sc}_g\geq 0$ may not imply ${\rm Sc}^{{\rm vol}_n}(M^n)\geq 0$. There are a lot of scalar-flat metrics but not strongly scalar flat metrics, i.e., ${\rm Sc}_g\equiv 0$ but not ${\rm Ricc}_g\neq 0$. For instance, the product metric on $S^2(1)\times \Sigma$, where $\Sigma$ is a closed hyperbolic surface, is the scalar-flat metric, but not the Ricci-flat metric. For those metrics, we have
 \begin{gather*}
 {\rm dVol}_g(B_r(x)) = {\rm vol}_E(B_r(\mathbf{R}^n))\left[1 + \frac{-3\|{\rm Rie}(x)\|^2_g+8\|{\rm Ricc}(x)\|^2_g}{360(n+2)(n+4)} r^4 + O\big(r^6\big)\right]
 \end{gather*}
 for $B_r(x) \subset M^n$ as $r\to 0$ {\rm \cite[Theorem~3.3]{MR521460}}. If $8\|{\rm Ricc}(x)\|^2_g>-3\|{\rm Rie}(x)\|^2_g$ for some point, then ${\rm Sc}_g\geq 0$ does not imply ${\rm Sc}^{{\rm vol}_n}(M^n)\geq 0$.
\end{Remark}

\begin{Theorem}
Assume that the mm-space $(X^n, d, \mu)$ satisfies $n$-dimensional condition and the curvature-dimension condition ${\rm CD}(\kappa, n)$ for $\kappa\geq 0$ and $n\geq 2$, then $(X^n, d, \mu)$ satisfies ${\rm Sc}^{{\rm vol}_n}(X^n)\geq n\kappa$.
\end{Theorem}

\begin{proof}
In fact, one only needs the generalized Bishop--Gromov volume growth inequality, which is implied by the curvature-dimension of $X^n$ \cite[Theorem~2.3]{MR2237207}.

 (i) If $\kappa=0$, then
 \begin{gather*}
 \frac{\mu(B_r(x))}{\mu(B_R(x))}\geq \left(\frac{r}{R}\right)^n
 \end{gather*}
 for all $ 0< r <R$. That is
\begin{gather*}
 \frac{\mu(B_r(x))}{{\rm vol}_E(B_r(\mathbf{R}^n))} = \frac{\mu(B_r(x))}{\alpha(n)r^n}\geq \frac{\mu(B_R(x))}{\alpha(n)R^n}=\frac{\mu(B_R(x))}{{\rm vol}_E(B_R(\mathbf{R}^n))},
\end{gather*}
where $\alpha(n)=\frac{{\rm vol}_E(B_r(\mathbf{R}^n))}{r^n}$.
Combining the $n$-dimensional condition,
\begin{gather*}
 \lim\limits_{r \to 0}\frac{\mu(B_r(x))}{{\rm vol}_E(B_r(\mathbf{R}^n))}=1,
\end{gather*}
 that implies ${\rm Sc}^{{\rm vol}_n}(X)\geq 0$.

(ii) If $\kappa >0$, then
\begin{gather*}
\frac{\mu(B_r(x))}{\mu(B_R(x))}\geq \frac{\int_0^r \big[\sin\big(t\sqrt{\frac{\kappa}{(n-1)}}\big)\big]^{n-1}{\rm d}t}{\int_0^R \big[\sin\big(t\sqrt{\frac{\kappa}{(n-1)}}\big)\big]^{n-1}{\rm d}t}
\end{gather*}
 for all $ 0< r \leq R\leq \pi \sqrt{\frac{(n-1)}{\kappa}}$.

Since the scalar curvature of the product manifold $S^2(\gamma)\times \mathbf{R}^{n-2}$ is $n\kappa$, where $\gamma=\sqrt{\frac{2}{n\kappa}}$, then there exists $C_1, C_2 >0$ such that
\begin{gather*}
1 - \frac{n\kappa}{6(n+2)}r_1^2 - C_2r_1^4 \leq \widetilde{{\rm vol}_{S\times E}(B_{r_1}(y))} := \frac{ {\rm vol}_{S\times E}(B_{r_1}(y))}{{\rm vol}_E(B_{r_1}(\mathbf{R}^n))} \leq 1 - \frac{n\kappa}{6(n+2)}r_1^2 + C_2r_1^4,
\end{gather*}
 for $y\in S^2(\gamma)\times \mathbf{R}^{n-2}$ and $r_1\leq C_1$, where $C_1$, $C_2$ are decided by the product manifold $S^2(\gamma)\times \mathbf{R}^{n-2}$.

Let
\begin{gather*}
\widetilde{\mu(B_r(x))}:=\frac{\mu(B_r(x))}{{\rm vol}_E(B_r(\mathbf{R}^n))}
\end{gather*}
 and
 \begin{gather*}
 f(r):= \frac{\int_0^r \big[\sin\big(t\sqrt{\frac{\kappa}{(n-1)}}\big)\big] ^{n-1}{\rm d}t}{{\rm vol}_E(B_r(\mathbf{R}^n))},
 \end{gather*}
 then the generalized Bishop--Gromov inequality can be re-formulated as
 \begin{gather*}
 \widetilde{\mu(B_R(x))} \leq \widetilde{\mu(B_r(x))} \frac{f(R)}{f(r)}
 \end{gather*}
 for all $ 0< r <R\leq \pi \sqrt{\frac{(n-1)}{\kappa}}$. The asymptotic expansion of $f(r)$ is
\begin{gather*}
f(r)= \frac{\frac{1}{n}r^n\big[\frac{\kappa}{(n-1)}\big]^{\frac{(n-1)}{2}}-\frac{(n-1)}{6(n+2)}r^{n+2} \big[\frac{\kappa}{(n-1)}\big]^{\frac{n+1}{2}}+O\big(r^{n+4}\big)}{{\rm vol}_E(B_r(\mathbf{R}^n))}
\end{gather*}
as $ r\rightarrow 0$.
 Thus, the asymptotic expansion of $\frac{f(R)}{f(r)}$ is
 \begin{gather*}
 \frac{f(R)}{f(r)}= \frac{1-\frac{n\kappa}{6(n+2)}R^2+ O\big(R^4\big)}{1-\frac{n\kappa}{6(n+2)}r^2+ O\big(r^4\big)}
 \end{gather*}
 as $R \to 0$, $r \to 0$.
 The $n$-dimensional condition, $\lim\limits_{r\to 0} \widetilde{\mu(B_r(x))}=1$, implies that
 \begin{gather*}
 \widetilde{\mu(B_R(x))}\leq 1-\frac{n\kappa}{6(n+2)}R^2+ O\big(R^4\big)
 \end{gather*}
as $R\to 0$. Therefore, for any $\kappa'$ with $0<\kappa' < \kappa$, there exists $\epsilon_{\kappa'} >0$ such that for any $0<R\leq \epsilon_{\kappa'}$, we have
 \begin{gather*}
\widetilde{\mu(B_R(x))}<\widetilde{{\rm vol}_{S\times E}(B_{R}(y))},
 \end{gather*}
where $\widetilde{{\rm vol}_{S\times E}(B_{R}(y))}= \frac{ {\rm vol}_{S\times E}(B_{R}(y))}{{\rm vol}_E(B_{R}(\mathbf{R}^n))}$ is defined as before, the balls $B_{R}(y)$ are in $S^2(\gamma)\times \mathbf{R}^{n-2}$ and $\gamma=\sqrt{\frac{2}{n\kappa'}}$.
That is $X^n <_{\rm vol} S^2(\gamma)\times \mathbf{R}^{n-2}$, for all $\gamma>\sqrt{\frac{2}{n\kappa}}$, i.e., ${\rm Sc}^{{\rm vol}_n}(X^n)\geq n\kappa$.

In fact, one has the classical Bishop inequality by adding the $n$-dimensional condition to the generalized Bishop--Gromov volume growth inequality. It means that
\begin{itemize}\itemsep=0pt
\item if $\kappa=0$, $\mu(B_R(x))\leq {\rm vol}_E(B_R(\mathbf{R}^n))$ for all $R>0$,
\item if $\kappa>0$, $\mu(B_R(x))\leq {\rm vol}_{S^n}\big(B_R\big(S^n_{\frac{\kappa}{n-1}}\big)\big)$ for $0< R\leq \pi \sqrt{\frac{(n-1)}{\kappa}}$.
\end{itemize}
In other words, if $\kappa=0$, $X^n\leq_{\rm vol}\mathbf{R}^n$. If $\kappa>0$, $X^n\leq_{\rm vol}S^n_{\frac{\kappa}{n-1}}$. We have $S^n_{\frac{\kappa}{n-1}}<_{\rm vol} S^2(\gamma)\times \mathbf{R}^{n-2}$ for all $\gamma>\sqrt{\frac{2}{n\kappa}}$. Then $X^n<_{\rm vol} S^2(\gamma)\times \mathbf{R}^{n-2}$ for all $\gamma>\sqrt{\frac{2}{n\kappa}}$.

 Thus, we also get ${\rm Sc}^{{\rm vol}_n}(X^n)\geq n\kappa$.
\end{proof}

\begin{Remark}
Hence the mm-space $(X^n, d, \mu)$ with ${\rm Sc}^{{\rm vol}_n}(X^n)\geq n\kappa$ includes the mm-spaces that satisfies $n$-dimensional condition and the generalized Bishop--Gromov volume growth inequality as stated in the proof, e.g., the mm-spaces with the Riemannian curvature condition ${\rm RCD}(\kappa,n)$ \cite{MR3205729} or with the measure concentration property ${\rm MCP}(\kappa,n)$~\cite{MR2341840}.
\end{Remark}

\begin{Question}Let ${\rm Al}^n(1)$ be an orientable compact $n$-dimensional Aleksandrov space with curvature $\geq 1$, then do all continuous maps $\phi$ from ${\rm Al}^n(1)$ to the sphere $S^n$ with standard metric of non-zero degree satisfy ${\rm Lip}(\phi)\geq C(n)$? Here ${\rm Lip}(\phi)$ is the Lipschitz constant of $\phi$, $\phi$ maps the boundary of ${\rm Al}^n(1)$ to a point in~$S^n$ and $C(n)$ is a constant depending only on the dimension~$n$.
\end{Question}

\begin{Question}Assume the compact mm-space $(X, d, \mu)$ satisfies the curvature-dimension condition ${\rm CD}(n-1, n)$, $n$-dimensional condition and the covering dimension is also $n$, then do all continuous maps $\phi$ from $(X, d, \mu)$ to the sphere $S^n$ with standard metric, where $\phi$ is non-trivial in the homotopy class of maps, satisfy ${\rm Lip}(\phi)\geq C_1(n)$, where $C_1(n)$ is a~constant depending only on~$n$?
\end{Question}

\begin{Remark}
The questions above are inspired by Gromov's spherical Lipschitz bounded theorem in \cite[Section~3]{MR3816521} and the results above. The finite covering dimension is equal to the cohomological dimension over integer ring~$\mathbb{Z}$ for the compact metric space according to the Alexandrov theorem. The best constant of~$C(n)$ and~$C_1(n)$ would be~1 if both questions have positive answers.
\end{Remark}

\begin{Proposition}[quadratic scaling] Assume the compact mm-space $(X^n, d, \mu)$ satisfies \linebreak ${\rm Sc}^{{\rm vol}_n}(X^n)\geq \kappa> 0$, then ${\rm Sc}^{{\rm vol}_n}(\lambda X^n)\geq \lambda ^{-2}\kappa> 0$ and $r_{\lambda X^n}=\lambda r_{X^n}$ for all $\lambda>0$, where $\lambda X^n:=(X^n, \lambda \cdot d,\lambda^n \cdot \mu)$.
 \end{Proposition}

\begin{proof}
 First, we will show that the $n$-dimensional condition is stable under scaling. Let $d':=\lambda\cdot d$, $\mu':=\lambda^n \cdot \mu$, $B'_r(x)$ be an $r$-ball in the $(X^n,d')$, and $B_r(x)$ be an $r$-ball in the $(X^n,d)$, then $B'_r(x)=B_{\frac{r}{\lambda}}(x)$ as the subset in the $X^n$. One has
 \begin{gather*}
 \lim\limits_{r \to 0}\frac{\mu'(B'_r(x))}{{\rm vol}_E(B_r(\mathbf{R}^n))}=\lim\limits_{r \to 0}\frac{\mu'(B_{\frac{r}{\lambda}}(x))}{{\rm vol}_E(B_r(\mathbf{R}^n))}=\lim\limits_{r \to 0}\frac{\lambda^n \cdot \mu(B_{\frac{r}{\lambda}}(x))}{\lambda^n \cdot {\rm vol}_E(B_{\frac{r}{\lambda}}(\mathbf{R}^n))}=1,
\end{gather*}
 then $\lambda X^n$ satisfies the $n$-dimensional condition.

 Since $\lambda \cdot \big(S^2(\gamma)\times \mathbf{R}^{n-2} \big)=\lambda \cdot S^2(\gamma)\times \lambda\cdot\mathbf{R}^{n-2}=S^2(\lambda \gamma)\times \lambda\cdot\mathbf{R}^{n-2}$, we have
 \begin{gather*}
 \lambda \cdot X^n<_{\rm vol}\lambda \cdot \big(S^2(\gamma)\times \mathbf{R}^{n-2} \big)=S^2(\lambda \gamma)\times \lambda\cdot\mathbf{R}^{n-2}
 \end{gather*}
 for all $\lambda\gamma> \sqrt{\frac{2}{\lambda\kappa}}$ and $0 < \epsilon \leq \lambda r_{X^n}$. That means ${\rm Sc}^{{\rm vol}_n}(\lambda X^n)\geq \lambda ^{-2}\kappa> 0$ and $r_{\lambda \cdot X^n}=\lambda r_{X^n}$.
 \end{proof}

We also have ${\rm Sc}^{{\rm vol}_n}(\lambda X^n)\geq 0$ $(>0)$, if ${\rm Sc}^{{\rm vol}_n}(X^n)\geq 0$ $(>0)$.

\begin{Remark}Since the $n$-dimensional condition and definition of $n$-volumic scalar curvature is locally defined, we have the following construction.
\begin{itemize}\itemsep=0pt
\item Global to local: Let the locally compact mm-space $(X^n, d, \mu)$ satisfy ${\rm Sc}^{{\rm vol}_n}(X^n)\geq \kappa\geq 0$ and $Y^n\subset X$ be an open subset. Then, if $(Y^n,d_Y)$ is a complete length space, $(Y^n,d_Y,\mu\llcorner_Y)$ satisfies ${\rm Sc}^{{\rm vol}_n}(Y^n)\geq \kappa\geq 0$ and $r_{Y^n}=r_{X^n}$. Where $d_Y$ is the induced metric of $d$ and $\mu\llcorner_Y$ is the restriction operator, namely, $\mu\llcorner_Y(A):=\mu(Y^n\cap A)$ for $A\subset X^n$.

\item Local to global: Let $\{Y^n_i\}_{i\in I}$ be a finite open cover of a locally compact mm-space $(X^n, d, \mu)$. Assume that $(Y^n_i,d_{Y_i})$ is a complete length space and $(Y^n_i,d_{Y_i},\mu\llcorner_{Y_i})$ satisfies ${\rm Sc}^{{\rm vol}_n}(Y^n_i)\geq \kappa\geq 0$, then $(X^n, d, \mu)$ satisfies ${\rm Sc}^{{\rm vol}_n}(X^n)\geq \kappa\geq 0$ and $r_{X^n}$ can be chosen as a partition of unity of the functions $\{r_{Y^n_i}\}_{i\in I}$.
\end{itemize}
\end{Remark}

\begin{Question}Assume that ${\rm Sc}^{{\rm vol}_{n_1}}(X^{n_1}_1)\geq \kappa_1 (\geq 0)$ for the compact mm-space $(X^{n_1}_1, d_1, \mu_1)$ and ${\rm Sc}^{{\rm vol}_{n_2}} \big(X^{n_2}_2 \big)\geq \kappa_2 (\geq 0)$ for the compact mm-space $(X^{n_2}_2, d_2, \mu_2)$, then do we have
 \begin{gather*}
 {\rm Sc}^{{\rm vol}_{n_1+n_2}} \big(X^{n_1}_1\times X^{n_2}_2 \big)\geq \kappa_1+\kappa_2, \qquad r_{X^n_1 \times X^n_2}=\min \{r_{X^n_1},r_{X^n_2}\}
 \end{gather*}
 for $ \big(X^{n_1}_1\times X^{n_2}_2, d_3, \mu_3 \big)$? Here $X^{n_1}_1 \times X^{n_2}_2$ is endowed with the measure $\mu_3:= \mu_1 \otimes \mu_2$ and with the Pythagorean product metric $d_3 :=\sqrt{d_1^2+d_2^2}$.
\end{Question}

\section{smGH-convergence}\label{sm}
 Let $\{\mu_n\}_{n\in\mathbf{N}}$ and $\mu$ be Borel measures on the space $X$, then the sequence $\{\mu_n\}_{n\in\mathbf{N}}$ is said to converge \textit{strongly} (also called setwise convergence in other literature ) to a limit $\mu$ if $\lim\limits_{n\rightarrow \infty} \mu_n(\mathcal{A})=\mu(\mathcal{A})$ for every $\mathcal{A}$ in the Borel $\sigma$-algebra.

 A map $f\colon X\rightarrow Y$ is called an $\epsilon$-isometry between compact metric spaces $X$ and $Y$, if $|d_{X}(a, b)-d_{Y}(f(a), f(b))|\leq\epsilon$ for all $a,b \in X$ and it is almost surjective, i.e., for every $y\in Y$, there exists an $x\in X$ such that $d_Y(f(x), y)\leq \epsilon$.

 In fact, if $f$ is an $\epsilon$-isometry $X\rightarrow Y$, then there is a $(4\epsilon)$-isometry $f'\colon Y\rightarrow X$ such that for all $x\in X$, $y\in Y$, $d_X(f'\circ f(x), x)\leq 3\epsilon$, $d_Y(f\circ f'(y), y)\leq \epsilon$.

 \begin{Definition}[smGH-convergence]
 Let $(X_i, d_i, \mu_i)_{i\in\mathbf{N}}$ and $(X, d, \mu)$ be compact mm-spaces. $X_i$ converges to $X$ in the strongly measured Gromov--Hausdorff topology (smGH-convergence) if there are measurable $\epsilon_i$-isometries $f_i\colon X_i\rightarrow X$ such that $ \epsilon_i \rightarrow 0$ and $f_{i*}\mu_i\rightarrow \mu$ in the strong topology of measures as $i\rightarrow \infty$.

 If the spaces $(X_i,d_n,\mu_i,p_i)_{i\in\mathbf{N}}$ and $(X,d,\mu,p)$ are locally compact pointed mm-spaces, it is said that $X_i$ converges to $X$ in the pointed strongly measured Gromov--Hausdorff topo\-lo\-gy (psmGH-convergence) if there are sequences $r_i\to \infty$, $\epsilon_i \to 0$, and measurable pointed $\epsilon_i$-isometries $f_i\colon B_{r_i}(p_i)\to B_{r_i}(p)$, such that $f_{i*}\mu_i \to \mu$, where the convergence is strong convergence.
 \end{Definition}

\begin{Remark}Let $(X_i, d_i, \mu_i)_{i\in\mathbf{N}}$ converge to $(X, d, \mu)$ in the measured Gromov--Hausdorff topo\-lo\-gy, then there are measurable $\epsilon_i$-isometries $f_i\colon X_i\rightarrow X$ such that~$f_{i*}\mu_i$ weakly converges to~$\mu$. If there is a Borel measure $\nu$ on $X$ such that $\sup\limits_{i}f_{i*}\mu_i\leq \nu$, i.e., $\sup \limits_{i}f_{i*}\mu_i (\mathcal{A})\leq \nu(\mathcal{A})$ for every~$\mathcal{A}$ in the Borel $\sigma$-algebra on~$X$, then $X_i$ smGH-converges to $X$ (see \cite[Lemma~4.1]{MR1453465}).
\end{Remark}

\begin{Remark}The $n$-dimensional condition is not preserved by the measured Gromov--Hausdorff convergence as the following example shows. Let $\big\{a_i S^2:= \big(S^2, a_i d_S \big)\big\}$ $(a_i \in (0,1))$ be a sequence of space, then the limit of $a_i S^2$ under the measured Gromov--Hausdorff convergence is a point when $a_i$ goes to $0$. The limit exists as the Ricci curvature of~$a_i S^2$ is bounded below by~$1$.
\end{Remark}

\begin{Remark}
The $n$-dimensional condition is not preserved by the smGH-convergence since the limits of $\lim\limits_{r \to 0}\lim\limits_{i \to \infty}\frac{\mu_i(B_r(x))}{{\rm vol}_E(B_r(\mathbf{R}^n))}$ may not be commutative for some mm-spaces $(X, d, \mu_i)$. Assume the total variation distance of the measures goes to 0 as $i \to \infty$, i.e.,
\begin{gather*}
d_{TV}(\mu_i, \mu):=\sup\limits_{\mathcal{A}}|\mu_i(\mathcal{A})-\mu(\mathcal{A})|\to 0,
\end{gather*}
where $\mathcal{A}$ runs over the Borel $\sigma$-algebra of $X$, then the limits are commutative.
\end{Remark}

One can also define \textit{the total variation Gromov--Hausdorff convergence $($tvGH-convergence$)$} for mm-spaces by replacing the strong topology with the topology induced by the total variation distance in definition of smGH-convergence. Then tvGH-convergence implies smGH-convergence and the $n$-dimensional condition is preserved by tvGH-convergence.

\begin{Theorem}[stability]
If compact mm-spaces $(X_i^n, d_i, \mu_i)$ with ${\rm Sc}^{{\rm vol}_n}(X^n_i)\geq \kappa\geq 0$, SC-radius $r_{X^n_i}\geq R>0$, and $(X_i^n, d_i, \mu_i)$ smGH-converge to the compact mm-space $(X^n, d, \mu)$ with $n$-dimensional condition, then $X^n$ also satisfies ${\rm Sc}^{{\rm vol}_n}(X^n)\geq \kappa$ and the SC-radius $r_{X^n}\geq R$.
\end{Theorem}

\begin{proof}
Fix an $x\in X^n$ and let $B_r(x)$ be the small $r$-ball on $X^n$ where $r < R$, then there exists $x_i\in X^n_i$ such that $f^{-1}_i(B_r(x))\subset B_{r+4\epsilon_i}(x_i) $ where $B_{r+4\epsilon_i}(x_i) \subset X^n_i$ and $r+4\epsilon_i\leq R$. Thus, $f_{i*}\mu_i(B_r(x))\leq \mu_i(B_{r+4\epsilon_i}(x_i))$.
\begin{itemize}\itemsep=0pt
\item For $\kappa=0$, since ${\rm Sc}^{{\rm vol}_n}(X^n_i)\geq 0$ and SC-radius$\geq R>0$, then $\mu_i(B_r(x_i))\leq {\rm vol}_E(B_r(\mathbf{R}^n))$ for all $0<r \leq R$ and all $i$. Therefore, $f_{i*}\mu_i(B_r(x))< \mu_i(B_{r+4\epsilon_i}(x_i))\leq {\rm vol}_E(B_{r+4\epsilon_i}(\mathbf{R}^n))$ for $r+4\epsilon_i\leq R$. Since $\epsilon_i$ that is not related to $r$ can be arbitrarily small, then $\mu(B_r(x))\leq {\rm vol}_E(B_r(\mathbf{R}^n))$.
\item For $\kappa>0$, we have $\mu_i(B_r(x_i))< {\rm vol}_{S \times E} \big(B_r \big(S^2(\gamma)\times \mathbf{R}^{n-2} \big) \big)$ for all $0< r \leq R$, all $i$, and $\gamma>\sqrt{\frac{2}{\kappa}}$. Thus, $f_{i*}\mu_i(B_r(x))< \mu_i(B_{r+4\epsilon_i}(x_i))< {\rm vol}_{S\times E} \big(B_{r+4\epsilon_i} \big(S^2(\gamma)\times \mathbf{R}^{n-2} \big) \big)$ for $r+4\epsilon_i\leq R$. Since $\epsilon_i$ that is not related to $r$ can be arbitrarily small, then $\mu(B_r(x))\leq {\rm vol}_{S\times E} \big(B_{r} \big(S^2(\gamma)\times \mathbf{R}^{n-2} \big) \big)$ for $\gamma>\sqrt{\frac{2}{\kappa}}$. Thus, $\mu(B_r(x))< {\rm vol}_{S\times E} \big(B_{r} \big(S^2\big(\sqrt{\frac{2}{\kappa+\epsilon'}}\big)\times \mathbf{R}^{n-2} \big) \big)$, where $0<\epsilon'$ is independence on~$r$ and~$\epsilon'$ can as small as we want. Therefore, we have ${\rm Sc}^{{\rm vol}_n}(X^n)\geq \kappa$.\hfill\qed
\end{itemize}\renewcommand{\qed}{}
\end{proof}

\begin{Definition}[tangent space]
The mm-space $(Y,d_Y,\mu_Y, o)$ is a tangent space of $(X^n,d,\mu)$ at $p\in X^n$ if there exists a sequence $\lambda_i\to \infty$ such that $(X^n,\lambda_i\cdot d,\lambda_i^n \cdot\mu,p)$ psmGH-converges to $(Y,d_Y,\mu_Y, o)$ as $\lambda_i\to \infty$.
\end{Definition}

Therefore, $(Y,d_Y,\mu_Y, o)$ also satisfies the $n$-dimensional condition and can be written as $Y^n$.

\begin{Corollary}Assume the compact mm-space $(X^n,d,\mu)$ with ${\rm Sc}^{{\rm vol}_n}(X^n)\geq\kappa\geq 0$ and the tangent space $(Y^n,d_Y,\mu_Y, o)$ of $X^n$ exists at the point $p$, then $(Y^n,d_Y,\mu_Y, o)$ satisfies ${\rm Sc}^{{\rm vol}_n}(Y^n)\allowbreak\geq 0$ and the SC-radius$\geq r_{X^n}$.
\end{Corollary}

\begin{proof}
Since the $n$-volumic scalar curvature has the quadratic scaling property, i.e., ${\rm Sc}^{{\rm vol}_n}(\lambda X^n)\allowbreak \geq \lambda ^{-2}\kappa \geq 0$ and $r_{\lambda X^n}=\lambda r_{X^n}$ for all $\lambda>0$, where $\lambda X^n:=(X^n, \lambda \cdot d,\lambda^n \cdot \mu)$, then ${\rm Sc}^{{\rm vol}_n}(Y^n)\geq 0$ is implied by the stability theorem.
\end{proof}

The mm-spaces with ${\rm Sc}^{{\rm vol}_n}\geq 0$ includes some of the Finsler manifolds, for instance, $\mathbf{R}^n$~equip\-ped with any norm and with the Lebesgue measure satisfies ${\rm Sc}^{{\rm vol}_n}\geq 0$ and any smooth compact Finsler manifold is a ${\rm CD}(\kappa, n)$ space for appropriate finite $\kappa$ and~$n$~\cite{MR2546027}. It is well-known that Gigli's infinitesimally Hilbertian~\cite{MR3381131} can be seen as the Riemannian condition in ${\rm RCD}(\kappa, n)$ space. Thus, infinitesimally Hilbertian can also be used as a Riemannian condition in the mm-spaces with ${\rm Sc}^{{\rm vol}_n}\geq 0$.

\begin{Definition}[${\rm RSC}(\kappa, n)$ space]
The compact mm-space $(X^n,d,\mu)$ with the $n$-dimensional condition is a Riemannian $n$-volumic scalar curvature$\geq\kappa$ space $({\rm RSC}(\kappa, n)$ space$)$ if it is infinitesimally Hilbertian and satisfies the ${\rm Sc}^{{\rm vol}_n}(X^n)\geq\kappa\geq 0$.
\end{Definition}

Note that any finite-dimensional Alexandrov spaces with curvature bounded below are infinitesimally Hilbertian. Then
\begin{gather*}
{\rm Al}^n(\kappa)\Rightarrow {\rm RCD}((n-1)\kappa, n)\Rightarrow {\rm RSC}((n(n-1)\kappa, n)
\end{gather*}
 on $(X^n,d,\mathcal{H}^n)$, where the measure $\mathcal{H}^n$ is the $n$-dimensional Hausdorff measure that satisfies the $n$-dimensional condition.

\begin{Question} Are ${\rm RSC}(\kappa, n)$ spaces stable under tvGH-convergence?
\end{Question}

\begin{Remark}[convergence of compact mm-spaces]
For the compact metric measure spaces with probability measures, one can consider mGH-convergence, Gromov--Prokhorov convergence, Gromov--Hausdorff--Prokhorov convergence, Gromov--Wasserstein convergence, Gromov--Hausdorff--Wasserstein convergence, Gromov's
$\underline{\Box}$-convergence, Sturm's $\mathbb{D}$-convergence \cite[Section~27]{MR2459454}, and Gromov--Hausdorff-vague convergence \cite{MR3522292}. smGH-convergence implies those convergences for compact metric measure spaces with probability measures, since the measures converge strongly in smGH-convergence and converge weakly in other situations.

Note that mm-spaces with infinitesimally Hilbertian are not stable under mGH-conver\-gen\-ce~\cite{MR3381131}. It is not clear if the infinitesimally Hilbertian are preserved under smGH-convergence or tvGH-convergence.
\end{Remark}

\section[Smooth mm-space with Sc(alpha,beta)>0]{Smooth mm-space with $\boldsymbol{{\rm Sc}_{\alpha, \beta}>0}$}\label{smooth mm-space}
 Let the smooth metric measure space $\big(M^n, g, {\rm e}^{-f}\,{\rm dVol}_g\big)$ (also known as the weighted Riemannian manifold in some references), where $f$ is a $C^2$-function on $M^n$, $g$ is a $C^2$-Riemannian metric and $n \geq 2$, satisfy the curvature-dimension condition ${\rm CD}(\kappa, n)$ for $\kappa\geq 0$, then $M^n$ also satisfies ${\rm Sc}^{{\rm vol}_n}(M^n)\geq n\kappa$.

 Motivated by the importance of the Ricci Bakry-Emery curvature, i.e.,
 \begin{gather*}
 {\rm Ric}^M_f= {\rm Ricc} +{\rm Hess}(f),
 \end{gather*}
 the weighted sectional curvature of smooth mm-space was proposed and discussed in~\cite{MR3397488}. On the other hand, Perelman defined and used the P-scalar curvature in his $\mathcal{F}$-functional in \cite[Section~1]{2002math.....11159P}. Inspired by the P-scalar curvature, i.e., $ {\rm Sc}_g + 2\bigtriangleup_gf - \|\bigtriangledown_gf\|^2_g$, we propose another scalar curvature on the smooth mm-space.

 \begin{Definition}[weighted scalar curvature ${\rm Sc}_{\alpha, \beta}$]
 The weighted scalar curvature ${\rm Sc}_{\alpha, \beta}$ on the smooth mm-space $\big(M^n, g, {\rm e}^{-f}\,{\rm dVol}_g\big)$ is defined by
 \begin{gather*}
 {\rm Sc}_{\alpha, \beta}:= {\rm Sc}_g + \alpha \bigtriangleup_gf - \beta\| \bigtriangledown_gf\|^2_g.
 \end{gather*}
 \end{Definition}
 Note that the Laplacian $\bigtriangleup_g$ here is the trace of the Hessian and ${\rm Sc}^{{\rm vol}_n}(M^n)\geq \kappa\geq 0$ is equivalent to ${\rm Sc}_{\alpha, \beta}\geq \kappa\geq 0$ for $\alpha=3$ and $\beta=3$ (see \cite[Theorem~8]{MR1458581} or the proof of Corollary~\ref{weighted volume} below).
 \begin{Example}\quad
 \begin{enumerate}\itemsep=0pt
 \item[1.] For $\alpha=\frac{2(n-1)}{n}$ and $\beta=\frac{(n-1)(n-2)}{n^2}$, the ${\rm Sc}_{\frac{2(n-1)}{n}, \frac{(n-1)(n-2)}{n^2}}$ is the Chang--Gursky--Yang's conformally invariant scalar curvature for the smooth mm-space \cite{MR2203156}. That means for a~$C^2$-smooth function $w$ on $M^n$, one has
 \begin{align*}
 {\rm Sc}_{\frac{2(n-1)}{n}, \frac{(n-1)(n-2)}{n^2}}\big({\rm e}^{2w}g\big)={\rm e}^{-2w}{\rm Sc}_{\frac{2(n-1)}{n}, \frac{(n-1)(n-2)}{n^2}}(g).
\end{align*}
\item[2.] For $\alpha=2$ and $\beta=\frac{m+1}{m}$, where $m\in \mathbb{N}\cup \{0, \infty\}$, the ${\rm Sc}_{2, \frac{m+1}{m}}$ is Case's weighted scalar curvature and Case also defined and studied the weighted Yamabe constants in~\cite{MR3116020}. Case's weighted scalar curvature is the classical scalar curvature if $m=0$. If $m=\infty$, then it is Perelman's $P$-scalar curvature.
 \end{enumerate}
\end{Example}
Note that the results in this paper are new for those examples.

 \subsection[Spin manifold and Sc(alpha,beta)>0]{Spin manifold and $\boldsymbol{{\rm Sc}_{\alpha, \beta}> 0}$}
 For an orientable closed surface with density $\big(\Sigma, g, {\rm e}^{-f}\,{\rm dVol}_g\big)$ with ${\rm Sc}_{\alpha, \beta} > 0$ and $\beta\geq 0$, then the inequality,
\begin{gather*}
0 < \!\int_\Sigma\! {\rm Sc}_{\alpha, \beta} \,{\rm dVol}_g = \!\int_\Sigma\! \big({\rm Sc}_g + \alpha\bigtriangleup_gf - \beta\|\bigtriangledown_gf\|_g^2\big) \,{\rm dVol}_g
 = 4\pi\chi(\Sigma) - \beta\!\int_\Sigma\! \|\bigtriangledown_gf\|_g^2 \,{\rm dVol}_g,\!
\end{gather*}
implies that $\chi(\Sigma)> 0$. Thus, $\Sigma$ is a 2-sphere.

The following proposition of vanishing harmonic spinors is owed to Perelman essentially and the proof is borrowed from \cite[Proposition~1]{MR3625165}.

\begin{Proposition}[vanishing harmonic spinors]\label{vanishing}
Assume the smooth mm-space $\big(M^n{,} g{,} {\rm e}^{-f} {\rm dVol}_g\big)$ is closed and spin. If $\alpha \in \mathbb{R}$, $\beta \geq \frac{|\alpha|^2}{4}$ and ${\rm Sc}_{\alpha, \beta}>0$, then the harmonic spinor of $M^n$ vanishes.
\end{Proposition}

\begin{proof}Let $\psi$ be a harmonic spinor, though the Schr\"odinger--Lichnerowicz--Weitzenboeck formula
\begin{align*}
\mathbb{D}^2=\bigtriangledown^*\bigtriangledown + \frac{1}{4}{\rm Sc}_g,
\end{align*}
 one has
\begin{align*}
0& = \int_M \left[\|\bigtriangledown_g\psi\|^2_g + \frac{1}{4}\big({\rm Sc}_{\alpha, \beta} - \alpha \bigtriangleup_g f + \beta\|\bigtriangledown_gf\|^2_g\big)\|\psi\|^2_g \right] {\rm dVol}_g\\
& = \int_M \left[\|\bigtriangledown_g\psi\|^2_g + \left(\frac{1}{4}{\rm Sc}_{\alpha, \beta} + \frac{\beta}{4}\|\bigtriangledown_g f\|^2_g\right)\|\psi\|^2_g + \frac{\alpha}{4} \big\langle \bigtriangledown_gf, \bigtriangledown_g\|\psi\|^2_g\big\rangle_g\right] {\rm dVol}_g .
\end{align*}
Then one gets
\begin{align*}
\frac{|\alpha|}{4}|\big\langle\bigtriangledown_g f,\bigtriangledown_g\|\psi\|^2_g \big\rangle_g| & \leq\frac{|\alpha|}{4}\big(c\|\bigtriangledown_gf\|_g \|\psi\|_g \times 2c^{-1}\|\bigtriangledown_g\psi\|_g\big)\\
& \leq \frac{|\alpha| c^2}{8} \|\bigtriangledown_gf\|^2_g \|\psi\|^2_g + \frac{c^{-2}|\alpha|}{2}\|\bigtriangledown_g\psi\|^2_g.
\end{align*}

Therefore,
\begin{equation*}
0 \geq \int_M \left[\left(1- \frac{c^{-2} |\alpha|}{2}\right)\|\bigtriangledown_g\psi\|^2_g + \frac{2\beta - c^{2} |\alpha|}{8} \|\bigtriangledown_gf\|^2_g\|\psi\|^2_g + \frac{1}{4} {\rm Sc}_{\alpha, \beta} \|\psi\|^2_g \right]  {\rm dVol}_g ,
\end{equation*}
where $c\neq 0$.
If $c^{-2} |\alpha| \leq 2$, $\beta \geq \frac{c^{2} |\alpha|}{2}$ and ${\rm Sc}_{\alpha, \beta} > 0$, then $\psi = 0$. So the conditions $\alpha \in \mathbb{R}$ and $\beta \geq \frac{|\alpha|^2}{4}$ are needed
\begin{align*}
\frac{|\alpha|}{4}|\big\langle\bigtriangledown_gf, \bigtriangledown_g \|\psi\|^2_g \big\rangle_g| &\leq \frac{|\alpha|}{4}\big(\|\bigtriangledown_gf\|_g \|\psi\|_g \times 2 \|\bigtriangledown_g \psi\|_g\big)\\
& = \frac{|\alpha|}{2}\big( c_1 \|\bigtriangledown_gf\|_g\|\psi\|_g \times c_1^{-1} \|\bigtriangledown_g\psi\|_g\big) \\
& \leq \frac{|\alpha|}{4} \big(c_1^2 \|\bigtriangledown_gf\|^2_g \|\psi\|^2_g + c_1^{-2} \|\bigtriangledown_g\psi\|^2_g\big).
 \end{align*}

Thus,
\begin{equation*}
0 \geq \int_M \left[\left(1- \frac{c_1^{-2} |\alpha|}{4}\right)\|\bigtriangledown_g\psi\|^2_g + \frac{\beta - c_1^{2} |\alpha|}{4} \|\bigtriangledown_gf\|^2_g\|\psi\|^2_g + \frac{1}{4} {\rm Sc}_{\alpha, \beta} \|\psi\|^2_g\right]  {\rm dVol}_g ,
\end{equation*}
where $c_1 \neq 0$. If $c_1^{-2} |\alpha| \leq 4$, $\beta \geq c_1^{2} |\alpha|$ and ${\rm Sc}_{\alpha, \beta} > 0$, then $\psi = 0$. Also the conditions $\alpha \in \mathbb{R}$ and $\beta \geq \frac{|\alpha|^2}{4}$ are needed.
\end{proof}

The following 3 corollaries come from the proposition of vanishing of harmonic spinors.

\begin{Corollary}\label{Genus}
Assume the smooth mm-space $\big(M^n, g, {\rm e}^{-f}\,{\rm dVol}_g\big)$ is closed and spin. If $\alpha \in \mathbb{R}$, $\beta \geq \frac{|\alpha|^2}{4}$ and ${\rm Sc}_{\alpha, \beta}>0$, then the $\widehat{A}$-genus and the Rosenberg index of $M^n$ vanish.
\end{Corollary}

\begin{proof}Since the $C^*(\pi_1(M^n))$-bundle in the construction of the Rosenberg index \cite{MR720934} is flat, there are no correction terms due to curvature of the bundle. Then the Schr\"odinger--Lichne\-ro\-wicz--Weitzenboeck formula and the argument in the proof of vanishing harmonic spinors can be applied without change.
\end{proof}

\begin{Corollary}Assume that $M^n$ is a closed spin $n$-manifold and $f$ is a smooth function on~$M^n$. If one of the following conditions is met,
\begin{enumerate}\itemsep=0pt
\item[$(1)$] $N \subset M^n$ is a codimension one closed connected submanifold with trivial normal bundle, the inclusion of fundamental groups $\pi_1\big(N^{n-1}\big)\rightarrow\pi_1(M^n)$ is injective and the Rosenberg index of $N$ does not vanish, or
\item[$(2)$] $N\subset M^n$ is a codimension two closed connected submanifold with trivial normal bundle, $\pi_2(M^n)=0$, the inclusion of fundamental groups $\pi_1\big(N^{n-1}\big)\rightarrow\pi_1(M^n)$ is injective and the Rosenberg index of $N$ does not vanish, or
\item[$(3)$] $N=N_1\cap \dots \cap N_k$, where $N_1 \cdots N_k\subset M$ are closed submanifolds that intersect mutually transversely and have trivial normal bundles. Suppose that the codimension of~$N_i$ is at most two for all $i\in \{1\dot k \}$ and $\pi_2(N) \rightarrow \pi_2(M)$ is surjective and $\hat{A}(N)\neq 0$,
\end{enumerate}
 then $M^n$ does not admit a Riemnannian metric $g$ such that the smooth mm-space $\big(M^n, g, \allowbreak {\rm e}^{-f}\,{\rm dvol}_g\big)$ satisfies ${\rm Sc}_{\alpha, \beta}> 0$ for the dimension $n\geq 3$, $\alpha \in \mathbb{R}$ and $\beta \geq \frac{|\alpha|^2}{4}$.
\end{Corollary}

\begin{proof}The results in the \cite[Theorem~1.1]{MR3449594} and \cite[Theorem~1.9]{MR3704253} can be applied to show that the Rosenberg index of $M^n$ does not vanish and Corollary~\ref{Genus} implies the theorem.
\end{proof}

Let $\mathcal{R}_f(M^n):=\{(g,f)\}$ be the space of densities, where $g$ is a smooth Riemannian metric on~$M^n$ and~$f$ is a smooth function on $M^n$ and $\mathcal{R}^+_f(M^n)\subset \mathcal{R}_f(M^n)$ is the subspace of densities such that the smooth mm-space $\big(M^n, g, {\rm e}^{-f}\,{\rm dvol}_g\big)$ satisfies ${\rm Sc}_{\alpha, \beta}> 0$. Furthermore, let $\mathcal{R}^+_f(M^n)$ be endowed with the smooth topology.

\begin{Corollary}Assume $M^n$ is a closed spin $n$-manifold, $n\geq 3$, $\alpha \in \mathbb{R}$ and $\beta \geq \frac{|\alpha|^2}{4}$ and $\mathcal{R}^+_f(M^n)\neq \varnothing$, then there exists a homomorphism
\begin{gather*}
A_{m-1}\colon \ \pi_{m-1}(\mathcal{R}^+_f(M^n))\to KO_{n+m}
\end{gather*}
such that
\begin{itemize}\itemsep=0pt
\item $A_0\neq 0$, if $n\equiv 0,1$ $({\rm mod}~8)$,
\item $A_1\neq 0$, if $n\equiv -1,0$ $({\rm mod}~8)$,
\item $A_{8j+1-n}\neq 0$, if $n\geq 7$ and $8j-n\geq 0$.
\end{itemize}
\end{Corollary}

\begin{proof}Since the results in the \cite[Section~4.4]{MR358873} and~\cite{MR3073935}
depend on the existence of exotic spheres with non-vanishing $\alpha$-invariant. Let $\phi\colon M^n\to M^n$ be a diffeomorphism of $M^n$ and $(g, f)\in \mathcal{R}^+_f(M^n)$, then $(\phi^*g, f\circ\phi)$ is also in $\mathcal{R}^+_f(M^n)$. Combining it with Proposition \ref{vanishing} shows that Hitchin's construction of the map $A$ \cite[Proposition~4.6]{MR358873} can be applied to the case of $\mathcal{R}^+_f(M^n)$ and then we can finish the proof with the arguments in \cite[Section~4.4]{MR358873} and \cite[Section~2.5]{MR3073935}.
\end{proof}

\subsection{Conformal to PSC-metrics}

\begin{Proposition}[conformal to PSC-metrics]\label{Conformal}
Let $\big(M^n, g, {\rm e}^{-f}\,{\rm dVol}_g\big)$ be a closed smooth mm-space with ${\rm Sc}_{\alpha, \beta} > 0$. If the dimension $n\geq 3$, $\alpha \in \mathbb{R}$ and $\beta \geq \frac{(n-2)|\alpha|^2}{4(n-1)}$, then there is a metric $\tilde{g}$ conformal to $g$ with positive scalar curvature (PSC-metric).
\end{Proposition}

\begin{proof}One only needs to show for all nontrivial $u$, $\int_M - u L_gu \,{\rm dVol}_g > 0$ as in the Yamabe problem \cite{MR788292}, where
\begin{align*}
L_g: = \bigtriangleup_g - \frac{n-2}{4(n-1)}{\rm Sc}_g
\end{align*}
 is conformal Laplacian operator. To see this,
\begin{align*}
\int_M - u L_gu \,{\rm dVol}_g &= \int_M \big[ \|\bigtriangledown_gu\|_g^2 + \frac{n-2}{4(n-1)}{\rm Sc}_g u^2 \big] \,{\rm dVol}_g \\
& = \int_M \left[ \|\bigtriangledown_gu\|_g^2 + \frac{n-2}{4(n-1)}\big({\rm Sc}_{\alpha, \beta} - \alpha \bigtriangleup_gf + \beta\|\bigtriangledown_gf\|^2_g\big) u^2\right] {\rm dVol}_g \\
& = \int_M \left[ \|\bigtriangledown_gu\|_g^2 + \frac{n-2}{4(n-1)}({\rm Sc}_{\alpha, \beta} + \beta\|\bigtriangledown_gf\|^2_g)u^2\right.\\
&\left. \hphantom{=}{} + \frac{\alpha(n-2)}{2(n-1)} \langle\bigtriangledown_gf,\bigtriangledown_gu \rangle_g u \right] {\rm dVol}_g.
\end{align*}

Through the inequality
\begin{gather*}
 \langle\bigtriangledown_gf,\bigtriangledown_gu \rangle_g u \leq c_2 \|\bigtriangledown_gf\|_gu \times c_2^{-1} \|\bigtriangledown_gu\|_g \leq \frac{c_2^2 \|\bigtriangledown_gf\|_g^2 u^2 + c_2^{-2}\|\bigtriangledown_gu\|_g^2}{2},
\end{gather*}
one gets
\begin{gather*}
\int_M - u L_gu \,{\rm dVol}_g \geq
 \int_M \bigg[ \left(1- \frac{|\alpha|c_2^{-2}(n-2)}{4(n-1)}\right)\|\bigtriangledown_gu\|_g^2 \\
\hphantom{\int_M - u L_gu \,{\rm dVol}_g \geq}{}
+\frac{\big(\beta -|\alpha|c_2^{-2}\big)(n-2)}{4(n-1)}\|\bigtriangledown_gf\|_g^2u^2+ \frac{n-2}{4(n-1)}{\rm Sc}_{\alpha, \beta}u^2\bigg]{\rm dVol}_g,
\end{gather*}
where $c_2 \neq 0$.

If $ |\alpha|c_2^{-2} \leq \frac{4(n-1)}{n-2}$, $\beta \geq c_2^2|\alpha|$ and ${\rm Sc}_{\alpha, \beta}> 0$, then
\begin{align*}
 \int_M - u L_gu\, {\rm dVol}_g >0.
\end{align*}
 So the conditions $n>2$, $\alpha \in \mathbb{R}$ and $\beta \geq \frac{(n-2)\alpha^2}{4(n-1)}$ are needed.
\end{proof}

\begin{Remark}
 The proof was borrowed from \cite[Proposition~2]{MR3625165}. The two propositions above offer a geometric reason why the condition of the vanishing of $\widehat{A}$-genus (without simply connected condition) does not imply that $M^n$ can admit a PSC-metric for the closed spin manifold $M^n$.
\end{Remark}

The proposition of conformal to PSC-metrics has following 3 corollaries.

\begin{Corollary}[weighted spherical Lipschitz bounded]
Let $\big(M^n, g, {\rm e}^{-f}\,{\rm dVol}_g\big)$ be a closed orientable smooth mm-space with ${\rm Sc}_{\alpha, \beta}\geq \kappa >0$, $3\leq n\leq 8$, $\alpha \in \mathbb{R}$ and $\beta \geq \frac{(n-2)|\alpha|^2}{4(n-1)}$, then the Lipschitz constant of the continuous map $\phi$ from $\big(M^n, g, {\rm e}^{-f}\,{\rm dVol}_g\big)$ to the sphere $S^n$ with standard metric of non-zero degrees has uniformly non-zero lower bounded.
\end{Corollary}

\begin{proof}There is a metric $\tilde{g}$ conformal to $g$ with scalar curvature${}\geq n(n-1)$ by the proposition of conformal PSC-metrics. For the continuous map $\phi$ from $(M^n, \tilde{g})$ to $S^n$ of non-zero degrees, the Lipschitz constant of $\phi$ is greater than a constant that depends only on the dimensions $n$ by Gromov's spherical Lipschitz bounded theorem \cite[Section~3]{MR3816521}. Since the conformal function has the positive upper bound by the compactness of the manifold, then the Lipschitz constant has uniformly non-zero lower bounded.
\end{proof}

\begin{Corollary}\label{weighted volume}
For the closed smooth mm-space $\big(M^n, g, {\rm e}^{-f}\,{\rm dVol}_g\big)$ $(n\geq 3)$ with ${\rm Sc}^{{\rm vol}_n}(M^n) \allowbreak >0$, there is a metric $\hat{g}$ conformal to $g$ with PSC-metric. In particular, the $\widehat{A}$-genus and Rosenberg index vanish with additional spin condition.

For the closed orientable smooth mm-space $\big(M^n, g, {\rm e}^{-f}\,{\rm dVol}_g\big)$ $(3\leq n\leq 8)$ with ${\rm Sc}^{{\rm vol}_n}(M^n)\geq \kappa >0$, then the Lipschitz constant of the continuous map $\phi$ from $\big(M^n, g, {\rm e}^{-f}\,{\rm dVol}_g\big)$ to the sphere~$S^n$ with standard metric of non-zero degrees has uniformly non-zero lower bounded.
\end{Corollary}

\begin{proof}The volume of the small disk of $\big(M^n, g, {\rm e}^{-f}\,{\rm dVol}_g\big)$ was computed in \cite[Theorem~8]{MR1458581},
\begin{align*}
\mu(B_r(x)) = {\rm vol}_E(B_r(\mathbf{R}^n))\left[1 - \frac{{\rm Sc}_g +3 \bigtriangleup_gf -3\| \bigtriangledown_gf\|^2_g }{6(n+2)} r^2 + O\big(r^4\big)\right]
 \end{align*}
as $r \rightarrow 0$. Since ${\rm Sc}^{{\rm vol}_n}(M^n) > 0$, i.e., $ \mu(B_r(x)) < {\rm vol}_E(B_r)$ as $r \rightarrow 0$, then
\begin{align*}
 {\rm Sc}_g +3 \bigtriangleup_gf -3\|\bigtriangledown_gf\|^2_g > 0.
\end{align*}
 Therefore, the propositions of vanishing harmonic spinors and of conformal PSC-metrics and Corollary~\ref{Genus} imply it.
\end{proof}

\begin{Remark}
Since any weighted Riemannian manifold (with non-trivial Borel measure) is infinitesimally Hilbertian (see \cite{MR4059811}), Corollary \ref{weighted volume} also works for $\big(M^n, g, {\rm e}^{-f}\,{\rm dVol}_g\big)$ with ${\rm RSC}(\kappa, n)$ condition.
\end{Remark}

Enlargeability as an obstruction to the existence of a PSC-metric on a closed manifold was introduced by Gromov--Lawson. We call a manifold enlargeable as Gromov--Lawson's definition in \cite[Definition~5.5]{MR720933}

\begin{Corollary}
Assume $M^n$ $(n\geq 3)$ is a closed spin smooth enlargeable manifold, then $\mathcal{R}^+_f(M^n)$ is an empty set for $\alpha \in \mathbb{R}$ and $\beta \geq \frac{(n-2)|\alpha|^2}{4(n-1)}$.

 In particular, $\big(\mathbb{T}^n,g, {\rm e}^{-f}\,{\rm dVol}_g\big)$ does not satisfy ${\rm Sc}^{{\rm vol}_n}(\mathbb{T}^n) >0$ for any $C^2$-smooth Riemannian metrics~$g$ and $C^2$-smooth functions~$f$ on the torus~$\mathbb{T}^n$.
\end{Corollary}

\begin{proof}Since a closed enlargeable manifold cannot carry a PSC-metric \cite[Theorem~5.8]{MR720933}, Proposition \ref{Conformal} implies $\mathcal{R}^+_f(M^n)=\varnothing$ for $\alpha \in \mathbb{R}$ and $\beta \geq \frac{(n-2)|\alpha|^2}{4(n-1)}$.

 $\mathbb{T}^n$ is an important example of enlargeable manifolds and then Corollary \ref{weighted volume} implies that $\big(\mathbb{T}^n,g, {\rm e}^{-f}\,{\rm dVol}_g\big)$ does not satisfy ${\rm Sc}^{{\rm vol}_n}(\mathbb{T}^n) >0$ for $n\geq 3$. For dimension $2$, the conditions of ${\rm Sc}_{\alpha, \beta} > 0$ and $\beta\geq 0$ imply that the oriented surface is $2$-sphere.
\end{proof}

\subsection[f-minimal hypersurface and Sc(alpha,beta)>0]{$\boldsymbol{f}$-minimal hypersurface and $\boldsymbol{{\rm Sc}_{\alpha, \beta}>0}$}

In addition to using the Dirac operator method, Schoen--Yau's minimal hypersurface method~\cite{MR541332} is another main idea. For an immersed orientable hypersurface $N^{n-1}\subset M^n$, the weighted mean curvature vector~$H_f$ of~$N^{n-1}$ is defined by Gromov in \cite[Section~9.4.E]{MR1978494},
\begin{gather*}
H_f= H+ (\bigtriangledown_gf)^{\perp},
\end{gather*}
 where $H$ is the mean curvature vector field of the immersion, $(\cdot)^{\perp}$ is the projection on the normal bundle of $N^{n-1}$. The first and second variational formulae for the weighted volume functional of $N^{n-1}$ were derived in Bayle's thesis (also see~\cite{MR2342613}). We take the detailed presentation of such derivation for~\cite{MR3324919}. The $\big(N^{n-1}, \bar{g}\big)$ with the induced metric is called $f$-minimal hypersurface if the weighted mean curvature vector $H_f$ vanishes identically.

 In fact, the definition of $f$-minimal hypersurface can also be derived from the first variational formula. Furthermore, an $f$-minimal hypersurface is a minimal hypersurface of $(M^n, \tilde{g})$, where~$\tilde{g}$ is the conformal metric of $g$, $\tilde{g}={\rm e}^{-\frac{2f}{n-1}}g$.

The connection between the geometry of the ambient smooth mm-space and the $f$-minimal hypersurfaces occurs via the second variation of the weighted volume functional. For a hypersurface $\big(N^{n-1}, \bar{g}\big)$, the $L_f$ operator is defined by
\begin{gather*}
L_f:= \bigtriangleup_f + |A|^2 + {\rm Ricc}^M_f(\nu, \nu),
\end{gather*}
 where $\nu$ is the unit normal vector, $|A|^2$ denotes the square of the norm of the second fundamental form~$A$ of~$N^{n-1}$ and
\begin{gather*}
 \bigtriangleup_f:=\bigtriangleup_{\bar{g}} - \langle\bigtriangledown_{\bar{g}}f, \bigtriangledown_{\bar{g}}\cdot\rangle
\end{gather*}
is the weighted Laplacian. Through the second variational formula, a two-sided $f$-minimal hypersurface~$N^{n-1}$ is stable (called $L_f$-stable) if for any compactly supported smooth function $u\in C^{\infty}_c\big(N^{n-1}\big)$, it holds that
\begin{gather*}
 -\int_NuL_fu {\rm e}^{-f}\,{\rm dVol}_{\bar{g}}\geq 0.
\end{gather*}

\begin{Proposition}\label{hypersuface}
Let $\big(M^n, g, {\rm e}^{-f}\,{\rm dVol}_g\big)$ be a closed orientable smooth mm-space with \linebreak \mbox{${\rm Sc}_{\alpha, \beta} > 0$} and $\big(N^{n-1},\bar{g}\big)$ be the compact $L_f$-stable minimal hypersurface of $\big(M^n, g, {\rm e}^{-f}\,{\rm dVol}_g\big)$. If the dimension $n\geq 3$, $\alpha=2$, and $\beta \geq \frac{n-2}{n-1}$, then there exists a PSC-metric conformal to $\bar{g}$ on~$N^{n-1}$.
\end{Proposition}

\begin{proof}The $f$-minimal hypersurface $\big(N^{n-1}, \bar{g}\big)$ is $L_f$-stable if and only if $\big(N^{n-1}, \bar{\tilde{g}}\big)$ is stable as a minimal hypersurface on $(M^n, \tilde{g})$, where $\tilde{g}:={\rm e}^{-\frac{2f}{n-1}}g$ and $ \bar{\tilde{g}}$ is the induced metric of $\tilde{g}$ (see \cite[Appendix]{MR3324919}). On the other hand, the scalar curvature of $(M^n, \tilde{g})$ is
\begin{gather*}
{\rm Sc}_{\tilde{g}}= {\rm e}^{\frac{f}{n-1}}\left({\rm Sc}_g +2\bigtriangleup_gf-\frac{n-2}{n-1}\|\bigtriangledown_gf\|^2_g\right).
\end{gather*}
Thus, ${\rm Sc}_{\alpha, \beta} > 0$ with $n\geq 3$, $\alpha=2$, and $\beta \geq \frac{n-2}{n-1}$ imply ${\rm Sc}_{\tilde{g}}>0$. Then the standard Schoen--Yau's argument can be applied to show that $\bar{\tilde{g}}$ conformal to a PSC-metric on $N^{n-1}$.
\end{proof}

\begin{Remark}The minimal hypersurface method poses a stricter condition to the valid range of $\alpha, \beta$ than that of the Dirac operator method.
\end{Remark}

Since the oriented closed manifolds with a PSC-metric in 2 and 3 dimensions are classified by Gauess--Bonnet theorem and Perelman--Thurston geometrization theorem, then Proposition~\ref{hypersuface} can give the following elementary applications:

\begin{Corollary}\label{minisurface}
Let $\big(M^n, g, {\rm e}^{-f}\,{\rm dVol}_g\big)$ be a closed orientable smooth mm-space with ${\rm Sc}_{\alpha, \beta} > 0$.
\begin{enumerate}\itemsep=0pt
\item[$1.$] If $n=3$, $\alpha=2$, and $\beta \geq \frac{1}{2}$, then there is no closed immersed $L_f$-stable minimal $2$-di\-men\-sional surface with positive genus.
\item[$2.$] If $n=4$, $\alpha=2$, and $\beta \geq \frac{2}{3}$, then the closed immersed $L_f$-stable minimal $3$-dimensional submanifold must be spherical $3$-manifolds, $S^2\times S^1$ or the connected sum of spherical $3$-manifolds and copies of~$S^2\times S^1$.
\end{enumerate}
\end{Corollary}

\begin{Remark}[historical remark]The prototype of Corollary~\ref{minisurface}(1) is the Schoen--Yau's classic result, which said that the oriented closed 3-manifold with a PSC-metric has no compact immersed stable minimal surface of positive genus~\cite{MR541332}. The Schoen--Yau result had been generalized to Perelman's P-scalar curvature${}>0$ by Fan~\cite{MR2407091}. Note that one can also consider the noncompact immersed $L_f$-stable minimal 2-dimensional surface under the condition of Corollary~\ref{minisurface}(1) since an oriented complete stable minimal surface in a complete oriented 3-manifold with a PSC-metric is conformally equivalent to the complex plane~$\mathbb{C}$ showed by Fischer--Colbrie--Schoen~\cite{MR562550}.
\end{Remark}

The smooth mm-space with ${\rm Sc}_{\alpha, \beta}>0$ under suit ranges of $\alpha$ and $\beta$ implies the manifold admits PSC-metrics, but the manifold (itself) that can admit PSC-metrics does not necessarily imply there exists ${\rm Sc}_{\alpha, \beta}>0$.
\begin{Question}Does the smooth mm-space with ${\rm Sc}_{\alpha, \beta}>0$ under suitable ranges of $\alpha$ and $\beta$ give more topological restriction on the manifold than the PSC-metric on the manifold?
\end{Question}

\begin{Question}Let $M$ be a closed smooth manifold, $f$ be a smooth function on $M$ and $h$ be a smooth function that is negative for some point $p$ on $M$. What is the range of $\alpha$ and $\beta$ such that there exists a smooth Riemannian metric $g$ on $M$ satisfying
 \begin{gather*}
 {\rm Sc}_g + \alpha\bigtriangleup_gf - \beta\|\bigtriangledown_gf\|_g^2=h,
 \end{gather*}
 i.e., ${\rm Sc}_{\alpha, \beta}(g)=h$?
\end{Question}

Let $(M, g_i)$ be smooth Riemannian manifolds and $\{g_i\}_{i\in \mathbf{N}}$ $C^0$-converges to $g$, then $\{g_i\}_{i\in \mathbf{N}}$ also smGH-converges to $g$. Gromov showed that the scalar curvature$\geq \kappa$ is stable under $C^0$-convergence in \cite[Section~1.8]{MR3201312}.

\begin{Question}Assume smooth mm-spaces $\big(M^n,g_i, {\rm e}^{-f}\,{\rm dVol}_{g_i}\big)$ all satisfy ${\rm Sc}^{{\rm vol}_n}(M^n)\geq 0$ such that $\{g_i\}_{i\in \mathbf{N}}$ $C^2$-converges to the smooth Riemannian metric~$g$ on~$M^n$, then does $\big(M^n,g, \allowbreak{\rm e}^{-f}\,{\rm dVol}_{g}\big)$ also satisfy ${\rm Sc}^{{\rm vol}_n}(M^n)\geq 0$?
\end{Question}

\begin{Question}Let mm-spaces $\big(M^n,g, {\rm e}^{-f}\,{\rm dVol}_{g}\big)$ with ${\rm Sc}^{{\rm vol}_n}(M^n)\geq \kappa>0$, where~$M^n$ is a~closed smooth manifold, $g$ and~$f$ are $C^0$-smooth, then does there exist a PSC-metric on~$M^n$?
\end{Question}

Since the role of ${\rm Sc}_{\alpha, \beta}>0$ on the smooth mm-space is similar to the role of ${\rm Sc}>0$ on the manifold, one can try to extend the knowledge about ${\rm Sc}>0$ to ${\rm Sc}_{\alpha, \beta}>0$.

\subsection{Weighted rigidity}
Gromov's conjecture that said if a smooth Riemannian metric $g$ satisfies $g\geq g_{\rm st}$ and ${\rm Sc}(g)\geq {\rm Sc}(g_{\rm st})=n(n-1)$ on $S^n$ then $g=g_{\rm st}$, was proved by Llarull~\cite{MR1600027} and called Llarull rigidity theorem. A map $h\colon (M^n, g_M)\to (N^n, g_N)$ is said to be $\epsilon$-contracting if $\|h_*v\|_{g_N}\leq \epsilon\|v\|_{g_N}$ for all tangent vectors $v$ on $M^n$

\begin{Proposition}[weighted rigidity] Assume the smooth mm-space $\big(M^n,g, {\rm e}^{-f}\,{\rm dVol}_{g}\big)$ is closed and spin and there exists a smooth $1$-contracting map $h\colon (M^n, g)\to (S^n, g_{\rm st})$ of non-zero degree. If $\alpha \in \mathbb{R}$, $\beta \geq \frac{|\alpha|^2}{4}$ and ${\rm Sc}_{\alpha, \beta}\geq n(n-1)$, then~$h$ is an isometry between the metrics~$g$ and~$g_{\rm st}$. Furthermore, if $\alpha> 0$, then~$f$ is a constant function.
\end{Proposition}

\begin{proof}One just need to insert the tricks in the proof of Proposition~\ref{vanishing} to the proof in \cite[Theorem~4.1]{MR1600027}. Following the setup of Llarull, we only prove the even-dimensional~($2n$) case without loss of generality.

First, we will show that $h$ is an isometry. Fix $p\in M^{2n}$. Let $\{e_1,\dots, e_{2n}\}$ be a $g$-orthonormal tangent fame near $p$ such that $(\bigtriangledown_g e_k)_p=0$ for each~$k$. Let $\{\epsilon_1,\dots,\epsilon_{2n}\}$ be a $g_{\rm st}$-orthonormal tangent frame near $h(p)\in S^{2n}$ such that $(\bigtriangledown_{g_{\rm st}}\epsilon_k)_{h(p)}=0$ for each $k$. Moreover, the bases $\{e_1,\dots, e_{2n}\}$ and $\{\epsilon_1,\dots,\epsilon_{2n}\}$ can be chosen so that $\epsilon_j=\lambda_jh_*e_j$ for appropriate $\{\lambda_j\}_{j=1}^{2n}$. This is possible since $h_*$ is symmetric. Since $h$ is $1$-contracting map, $\lambda_k\geq 1$ for each $k$.

 Then one constructs the twisted vector bundles $S \bigotimes E$ over $M^{2n}$ as Llarull did. Let $R^E$ be the curvature tensor of $E$ and $\psi$ be a twisted spinor, then one gets
 \begin{gather*}
 \big\langle R^E\psi, \psi\big\rangle_g \geq -\frac{1}{4}\sum_{i\neq j}\frac{1}{\lambda_i\lambda_j}\|\psi\|_g.
\end{gather*}

For the twisted Dirac operator $\mathbb{D}_E$, one has $\mathbb{D}^2_E=\bigtriangledown^*\bigtriangledown + \frac{1}{4}{\rm Sc}_g + R^E$ and
\begin{gather*}
\int_M \big\langle \mathbb{D}^2_E \psi, \psi \big\rangle_g \,{\rm dVol}_g = \int_M \bigg[\|\bigtriangledown_g\psi\|^2_g + \frac{1}{4}\big({\rm Sc}_{\alpha, \beta} - \alpha \bigtriangleup_g f + \beta\|\bigtriangledown_gf\|^2_g\big)\|\psi\|^2_g \\
\hphantom{\int_M \big\langle \mathbb{D}^2_E \psi, \psi \big\rangle_g \,{\rm dVol}_g =}{} + \big\langle R^E\psi, \psi\big\rangle_g\bigg] \,{\rm dVol}_g\\
\hphantom{\int_M \big\langle \mathbb{D}^2_E \psi, \psi \big\rangle_g \,{\rm dVol}_g}{}
 = \int_M \bigg[\|\bigtriangledown_g\psi\|^2_g + \left(\frac{1}{4}{\rm Sc}_{\alpha, \beta} + \frac{\beta}{4}\|\bigtriangledown_g f\|^2_g\right)\|\psi\|^2_g \\
\hphantom{\int_M \big\langle \mathbb{D}^2_E \psi, \psi \big\rangle_g \,{\rm dVol}_g =}{}
 + \frac{\alpha}{4} \langle \bigtriangledown_gf, \bigtriangledown_g\|\psi\|^2_g\rangle_g + \big\langle R^E\psi, \psi\big\rangle_g\bigg] {\rm dVol}_g .
\end{gather*}

Because $\lambda_k\geq 1$ for each $k$, one gets
\begin{gather*}
 \big\langle R^E\psi, \psi\big\rangle_g \geq \frac{-2n(2n-1)}{4}\|\psi\|_g
 \end{gather*}
 and then
\begin{align*}
\frac{|\alpha|}{4}|\big\langle\bigtriangledown_gf, \bigtriangledown_g \|\psi\|^2_g \big\rangle_g| &\leq \frac{|\alpha|}{4}\big(\|\bigtriangledown_gf\|_g \|\psi\|_g \times 2 \|\bigtriangledown_g \psi\|_g\big)\\
& = \frac{|\alpha|}{2}\big( c_1 \|\bigtriangledown_gf\|_g\|\psi\|_g \times c_1^{-1} \|\bigtriangledown_g\psi\|_g\big) \\
& \leq \frac{|\alpha|}{4} \big(c_1^2 \|\bigtriangledown_gf\|^2_g \|\psi\|^2_g + c_1^{-2} \|\bigtriangledown_g\psi\|^2_g\big),
 \end{align*}
where $c_1 \neq 0$. Therefore,
\begin{gather*}
\int_M \big\langle \mathbb{D}^2_E \psi, \psi \big\rangle_g \,{\rm dVol}_g \geq
 \int_M \bigg[\left(1- \frac{c_1^{-2} |\alpha|}{4}\right)\|\bigtriangledown_g\psi\|^2_g + \frac{\beta - c_1^{2} |\alpha|}{4} \|\bigtriangledown_gf\|^2_g\|\psi\|^2_g \\
\hphantom{\int_M \langle \mathbb{D}^2_E \psi, \psi \rangle_g \,{\rm dVol}_g \geq}{}
 +\frac{1}{4} ({\rm Sc}_{\alpha, \beta}-2n(2n-1)) \|\psi\|^2_g \bigg] {\rm dVol}_g.
\end{gather*}

Furthermore, since $\alpha \in \mathbb{R}$, $\beta \geq \frac{|\alpha|^2}{4}$ and ${\rm Sc}_{\alpha, \beta}\geq 2n(2n-1)$, one can choose $c_1$ such that $c_1^{-2} |\alpha|\leq 4$, then $\beta - c_1^{2} |\alpha|\geq 0$. Thus,
\begin{equation*}
\int_M \big\langle \mathbb{D}^2_E \psi, \psi \big\rangle_g \,{\rm dVol}_g\geq \int_M \frac{1}{4} [{\rm Sc}_{\alpha, \beta}-2n(2n-1)]\|\psi\|^2_g \,{\rm dVol}_g\geq 0.
\end{equation*}

The fact ${\rm Index}(\mathbb{D}_{E^+})\neq 0$ implies ${\rm ker}(\mathbb{D}_E)\neq 0$ and then ${\rm Sc}_{\alpha, \beta}=2n(2n-1)$. Then using the inequality $\big\langle R^E\psi, \psi\big\rangle_g \geq -\frac{1}{4}\sum_{i\neq j}\frac{1}{\lambda_i\lambda_j}\|\psi\|_g$, one gets
\begin{equation*}
\int_M \big\langle \mathbb{D}^2_E \psi, \psi \big\rangle_g \,{\rm dVol}_g\geq \int_M \frac{1}{4} \bigg[\sum_{i\neq j}\left(1-\frac{1}{\lambda_i\lambda_j}\right)\bigg]\|\psi\|^2_g \,{\rm dVol}_g\geq 0.
\end{equation*}

Choosing $\psi\neq 0$ such that ${\mathbb D}_E\psi=0$, one has
\begin{align*}
0\leq 1-\frac{1}{\lambda_i\lambda_j}\leq 0
\end{align*}
for $i\neq j$. Thus, $\lambda_k=0$ for all $1\leq k\leq 2n$ and $h$ is an isometry.

Second, we will show that $f$ is a constant function. Since ${\rm Sc}_{\alpha, \beta}=2n(2n-1)$, ${\rm Sc}_g=2n(2n-1)$, $\alpha> 0$ and $\beta \geq \frac{|\alpha|^2}{4}$,
then $\bigtriangleup_g f\geq 0$. One has
\begin{align*}
 \int_M \bigtriangleup_g f \,{\rm dVol}_g=0
\end{align*}
for a closed manifold~$M^n$, so one gets $\bigtriangleup_g f=0$. That implies $\bigtriangledown_gf=0$ so that $f$ is a constant function on~$M^n$.
\end{proof}

\begin{Corollary}Let the closed and spin smooth mm-space $\big(M^n,g, {\rm e}^{-f}\,{\rm dVol}_{g}\big)$ with ${\rm Sc}^{{\rm vol}_n}(M^n)\allowbreak \geq n(n-1)$ and there exists a smooth $1$-contracting map $h\colon (M^n, g)\to (S^n, g_{\rm st})$ of non-zero degree, then~$h$ is an isometry between the metrics $g$ and $g_{\rm st}$.
\end{Corollary}

\begin{proof}Combining the weighted rigidity theorem and the proof of Corollary~\ref{weighted volume} can imply it.
\end{proof}

As Llarull rigidity theorem (and the weighted rigidity theorem) still holds if the condition that $h$ is $1$-contracting is replaced by the condition that $h$ is area-contracting, Gromov called such metrics area-extremal metrics, asked which manifolds possess area-extremal metrics, and conjectured that Riemannian symmetric spaces should have area-extremal metrics~\cite{MR1389019}, \cite[Section~17]{Gromov} and \cite[Section~4.2]{2019arXiv190810612G}. Goette--Semmelmann showed that several classes of symmetric spaces with non-constant curvatures are area-extremal \cite{MR1877585}.

\begin{Question}Can Goette--Semmelmann's results~{\rm \cite{MR1877585}} be generalized to the smooth mm-space with ${\rm Sc}_{\alpha, \beta}>0$ under other suitable conditions?
\end{Question}

Since Corollary~\ref{minisurface}(1) showed that the closed orientable immersed $L_f$-stable minimal $2$-dimensional surface in the closed orientable smooth mm-space $\big(M^n,g, {\rm e}^{-f}\,{\rm dVol}_{g}\big)$ with ${\rm Sc}_{2, \beta}>0$ $\big(\beta\geq \frac{1}{2}\big)$ is $2$-sphere, then one can consider rigidity of area-minimizing 2-sphere
in $3$-dimensional smooth mm-space. Bray’s volume comparison theorem \cite[Chapter~3, Theorem~18]{MR2696584} is another rigidity theorem that needs the conditions of Ricci curvature and scalar curvature bounded below. There are other rigidity phenomena involving scalar curvature, see~\cite{MR3076061}.

\begin{Question}Can Bray’s volume comparison theorem be extended to the smooth mm-space?
\end{Question}

\begin{Question}What is the correct Einstein field equation on the smooth mm-space?

If one replaces the Ricci and scalar curvature on the left hand side of Einstein field question by ${\rm Ricc}^M_f$ and ${\rm Sc}_{\alpha, \beta}$ for the smooth mm-space $\big(M^n,g, {\rm e}^{-f}\,{\rm dVol}_{g}\big)$, then what is the stress–energy tensor on the right hand side in this case?
\end{Question}

\subsection*{Acknowledgements}
 I am grateful to Thomas Schick for his help, the referees for their useful comments, and the funding from China Scholarship Council.

\pdfbookmark[1]{References}{ref}
\LastPageEnding

\end{document}